\documentclass{amsart}
\usepackage{amstext,amssymb,amsthm,amsopn,newlfont,graphpap,graphics,graphicx,mathrsfs,enumitem}\usepackage{amstext,amssymb,amsthm,amsopn,newlfont,graphpap,graphics,graphicx,mathrsfs,enumitem}
\allowdisplaybreaks
\usepackage[parfill]{parskip}
\usepackage[noadjust]{cite}
\usepackage{epigraph}
\usepackage[colorlinks=true,
            linkcolor=red,
            urlcolor=blue,
            citecolor=magenta]{hyperref}
\usepackage{color}
\usepackage{mathrsfs}
\allowdisplaybreaks
\theoremstyle{plain}
\newtheorem{thm}{Theorem}[section]
\newtheorem*{thm*}{Theorem}

\newtheorem*{prop*}{Proposition}

\newtheorem*{cor*}{Corollary}

\newtheorem*{lem*}{Lemma}
\theoremstyle{definition}
\newtheorem{defn}{Definition}[section]
\newtheorem*{defn*}{Definition}

\newtheorem*{exmp*}{Example}

\newtheorem*{exmps*}{Examples}

\newtheorem{rem}{Remark}[section]
\newtheorem*{rem*}{Remark}
\newtheorem{rems}{Remarks}[section]
\newtheorem*{rems*}{Remarks}

\newtheorem*{note*}{Note}
\newcommand{\N}{{\mathbb N}}
\newcommand{\Z}{{\mathbb Z}}
\newcommand{\R}{{\mathbb R}}
\newcommand{\C}{{\mathbb C}}
\newcommand{\F}{{\mathbb F}}

\DeclareMathOperator{\orb}{orb}

\begin{document}
\title[Chaoticity and spectrum of Rolewicz-type unbounded operators]
{On the chaoticity and spectral structure\\
of Rolewicz-type unbounded operators}
\author[Marat V. Markin]{Marat V. Markin}
\address{
Department of Mathematics\newline
California State University, Fresno\newline
5245 N. Backer Avenue, M/S PB 108\newline
Fresno, CA 93740-8001
}
\email{mmarkin@csufresno.edu}
\begin{abstract}
We prove the chaoticity and describe the spectral structure of Rolewicz-type weighted backward shift unbounded linear operators in the sequence spaces $l_p$ ($1\le p<\infty$) and $c_0$.
\end{abstract}
\maketitle
\epigraph{\textit{An old thing becomes new if you detach it from what usually surrounds it.}}{Robert Bresson}

\section[Introduction]{Introduction}

We consider the weighted backward shift unbounded linear operators:
\[
D(A)\ni x:=(x_k)_{k\in \N} \mapsto Ax:=(w^k x_{k+1})_{k\in \N},
\]
where $w\in \F$ ($\F:=\R$ or $\F:=\C$) such that $|w|>1$ is arbitrary, with the maximal domain
\[
D(A):=\left\{(x_k)_{k\in \N}\in X\,\middle|\,(w^k x_{k+1})_{k\in \N}\in X \right\}
\]
in the (real or complex) space $X:=l_p$ ($1\le p<\infty$) or $X:=c_0$, which naturally emerge from the classical Rolewicz weighted backward shift bounded linear operators
\[
X\ni x:=(x_k)_{k\in \N} \mapsto Ax:=w(x_{k+1})_{k\in \N},
\]
with $w\in \F$ such that $|w|>1$ \cite{Rolewicz1969} (see also \cite{Grosse-Erdmann-Manguillot}).

We show that each operator $A$ of this form is \textit{chaotic} and, provided the underlying space is complex, its spectrum $\sigma(A)=\C$, every $\lambda \in \C$ being an \textit{eigenvalue} for $A$ with an associated one-dimensional eigenspace.

\section[Preliminaries]{Preliminaries}

The notions of \textit{hypercyclicity} and \textit{chaoticity}, traditionally considered and well studied for \textit{continuous} linear operators on Fr\'echet spaces, in particular for \textit{bounded} linear operators on Banach spaces, and shown to be purely infinite-dimensional phenomena (see, e.g., \cite{Grosse-Erdmann-Manguillot,Guirao-Montesinos-Zizler,Rolewicz1969}), in \cite{B-Ch-S2001,deL-E-G-E2003} are extended 
to \textit{unbounded} linear ope\-rators in Banach spaces, where also found are sufficient conditions for unbounded hypercyclicity and certain examples of hypercyclic and chaotic unbounded linear differential operators. The chaoticity of general unbounded weighted backward shifts in the Bargmann space, which is a complex infinite-dimensional separable Hilbert space of entire functions, is characterized in \cite{Emam-Hesh2005}, where in particular, it is shown that the spectrum of such operators is the whole complex plane $\C$, with every $\lambda \in \C$ being an \textit{eigenvalue}. In \cite{Markin2018(7)}, (bounded or unbounded) \textit{scalar type spectral operators} in a complex Banach, in particular \textit{normal} ones in a complex Hilbert space, (see, e.g., \cite{Dunford1954,Survey58,Dun-SchI,Dun-SchII,Dun-SchIII,Plesner}) are proven to be \textit{non-hypercyclic}. 

\begin{defn}[Hypercyclicity and Chaoticity \cite{B-Ch-S2001,deL-E-G-E2003}] Let
\[
A:X\supseteq D(A)\to X
\]
be a (bounded or unbounded) linear operator in a (real or complex) Banach space $(X,\|\cdot\|)$. A nonzero vector 
\begin{equation}\label{idv}
x\in C^\infty(A):=\bigcap_{n=0}^{\infty}D(A^n)
\end{equation}
($A^0:=I$, $I$ is the \textit{identity operator} on $X$)
is called \textit{hypercyclic} if its \textit{orbit} under $A$
\[
\orb(x,A):=\left\{A^nx\right\}_{n\in\Z_+}
\]
($\Z_+:=\left\{0,1,2,\dots\right\}$ is the set of nonnegative integers) is dense in $X$.

Linear operators possessing hypercyclic vectors are
said to be \textit{hypercyclic}.

If there exist an $N\in \N$ ($\N:=\left\{1,2,\dots\right\}$ is the set of natural numbers) and a vector $x\in C^\infty(A)$ with $A^Nx = x$, such a vector is called a \textit{periodic point} for the operator $A$.

Hypercyclic linear operators with a dense set of periodic points are said to be \textit{chaotic}.
\end{defn}

\begin{rems}\
\begin{itemize}
\item In the definition of hypercyclicity, the underlying space is necessarily \textit{separable}.
\item If $x\in C^\infty(A)$ is a hypercyclic vector for a linear operator $A$ in a Banach space $(X,\|\cdot\|)$, then all vectors forming its orbit $\orb(x,A):=\left\{A^nx\right\}_{n\in\Z_+}$ are linearly independent and hypercyclic for $A$ (see, e.g., \cite{Grosse-Erdmann-Manguillot}).
\item The set $HC(A)$ of all hypercyclic vectors of a linear operator $A$, containing the dense orbit of its arbitrary hypercyclic vector, is dense in $(X,\|\cdot\|)$. The more so, dense in $(X,\|\cdot\|)$ is the subspace $C^\infty(A)\supset HC(A)$.
\end{itemize} 
\end{rems} 

In \cite{Rolewicz1969}, S. Rolewicz provides the first example of a hypercyclic bounded linear operator on a Banach space (see also \cite{Grosse-Erdmann-Manguillot,Guirao-Montesinos-Zizler}), which on the (real or complex) sequence space $l_p$ ($1\le p<\infty$) or $c_0$ (of vanishing sequences), the latter equipped with the supremum norm
\[
c_0\ni x:=(x_k)_{k\in \N}\mapsto \|x\|_\infty:=\sup_{k\in \N}|x_k|,
\] 
is the following weighted backward shift:
\[
A(x_k)_{k\in \N}:=w(x_{k+1})_{k\in \N},
\]
where $w\in \F$ ($\F:=\R$ or $\F:=\C$) such that $|w|>1$ is arbitrary. Rolewicz's operators also happen to be \textit{chaotic} \cite{Godefroy-Shapiro1991} (see also \cite{Grosse-Erdmann-Manguillot}). 

\section[Chaoticity and spectral structure of Rolewicz-type unbounded operators]{Chaoticity and spectral structure\\ of Rolewicz-type unbounded operators}

Here, we are to consider a natural unbounded extension
of the classical Rolewicz weighted backward shift operators.

\begin{thm}\label{Thm1} 
In the (real or complex) sequence space $X:=l_p$ ($1\le p<\infty$) or $X:=c_0$, the weighted backward shift
\[
D(A)\ni x:=(x_k)_{k\in \N} \mapsto Ax:=(w^k x_{k+1})_{k\in \N},
\]
where $w\in \R$ or $w\in \C$ such that $|w|>1$ is arbitrary, with the domain
\[
D(A):=\left\{(x_k)_{k\in \N}\in X\,\middle|\,(w^k x_{k+1})_{k\in \N}\in X \right\}
\]
is a chaotic unbounded linear operator.

Provided the underlying space $X$ is complex, its spectrum $\sigma(A)=\C$, every $\lambda \in \C$ being an \textit{eigenvalue} for $A$ with an associated one-dimensional eigenspace.
\end{thm}

\begin{proof} 
We let $\F:=\R$ or $\F:=\C$ and use  $\|\cdot\|$ to designate norm on $X$.

For each $n\in\N$, the linear operator
\[
A^n(x_k)_{k\in \N}=(w^{(k+n-1)+(k+n-2)+\dots+k} x_{k+n})_{k\in \N},\ (x_k)_{k\in \N}\in D(A^n),
\]
is \textit{densely defined}, the maximal domain 
\[
D(A^n):=\left\{(x_k)_{k\in \N}\in X\,\middle|\,(w^{(k+n-1)+(k+n-2)+\dots+k} x_{k+n})_{k\in \N}\in X \right\}
\]
containing the dense in $(X,\|\cdot\|)$ subspace $c_{00}$
of eventually zero sequences, and, as is easily seen,
is \textit{unbounded} and \textit{closed}.

Hence, the subspace 
\[
C^\infty(A):=\bigcap_{n=0}^{\infty}D(A^n)\supset c_{00}
\]
of all possible initial values for the orbits under $A$ is \textit{dense} in $(X,\|\cdot\|)$.

Observe that
\begin{equation}\label{incl1}
D(A^{n+1})\subseteq D(A^n),\ n\in\N.
\end{equation}

For the bounded right inverse of $A$:
\[
B(x_k)_{k\in \N}:=(0,w^{-1} x_1,w^{-2} x_2,\dots),\ (x_k)_{k\in \N}\in X,
\]
for any $n\in \N$, on $X$, we have:
\begin{equation}\label{rin}
B^n(x_k)_{k\in \N}=(\underbrace{0,\dots,0}_{\text{$n$ terms}},w^{-[n+(n-1)+\dots+1]}x_1,w^{-[(n+1)+n+\dots+2]}x_2,\dots).
\end{equation}

The subspace $c_{00}$ contains a countable dense in $(X,\|\cdot\|)$ subset
\begin{equation}\label{def2}
Y:=\left\{y^{(m)}:=\left(y^{(m)}_k\right)_{k\in \N}\right\}_{m\in \N}
\end{equation}
of all nonzero eventually zero sequences with rational (or complex rational) terms.

For each $m\in \N$, let
\begin{equation}\label{def1}
k_m:=\max\left\{k\in \N\,\middle|\,y^{(m)}_k\neq 0 \right\}.
\end{equation}

Setting $n_1:=1$, we can inductively build an increasing sequence
$(n_m)_{m\in \N}$ of natural numbers such that, for all $j,m\in \N$ with $m>j$,
\begin{equation}\label{est1}
\begin{split}
&n_m-n_j\ge \max(m,k_j)\quad \text{and}   \\
&|w|^{n_m+(n_m-1)+\dots+1}\ge |w|^{n_m+(n_m-1)+\dots+(n_m-n_j+1)+m}\left\|y^{(m)}\right\|.
\end{split}
\end{equation}

Let us show that the vector
\begin{equation}\label{def3}
x:=\sum_{m=1}^\infty B^{n_m}y^{(m)}
\end{equation}
is \textit{hypercyclic} for $A$, the above series converging in $(X,\|\cdot\|)$
since, for any $m=2,3,\dots$, in view of \eqref{rin} and \eqref{est1} with $j=1$,
\[
\left\|B^{n_m}y^{(m)}\right\|
\le |w|^{-[n_m+(n_m-1)+\dots+1]}\left\|y^{(m)}\right\|\le |w|^{-n_m-m}\le |w|^{-m}.
\]

Further, for each $k\in \N$,
\begin{multline*}
\sum_{m=1}^\infty A^{n_k}B^{n_m}y^{(m)}
=\sum_{m=1}^{k-1} A^{n_k-n_m}y^{(m)}+y^{(k)}
+\sum_{m=k+1}^\infty B^{n_m-n_k}y^{(m)}
\\
\hfill
\text{for $m=1,\dots,k-1$, by \eqref{est1}, $n_k-n_m\ge k_m$, by \eqref{def1}, $A^{n_k-n_m}y^{(m)}=0$;}
\\
\ \ \
=y^{(k)}+\sum_{m=k+1}^\infty B^{n_m-n_k}y^{(m)}.
\hfill
\end{multline*}

Since, for all $k,m\in \N$ with $m>k$, in view of \eqref{est1} with $j=k$,
\begin{equation}\label{est2}
\left\|B^{n_m-n_k}y^{(m)}\right\|
\le |w|^{-[(n_m-n_k)+(n_m-n_k-1)+\dots+1]}\left\|y^{(m)}\right\|\le |w|^{-m}.
\end{equation}

This implies that, for any $k\in \N$, the series
\[
\sum_{m=1}^\infty A^{n_k}B^{n_m}y^{(m)}
\]
converges in $(X,\|\cdot\|)$, and hence, by the \textit{closedness} of $A^{n_k}$ and in view of inclusion \eqref{incl1},
\[
x\in \bigcap_{k=1}^{\infty}D(A^{n_k})=C^\infty(A)
\]
and 
\[
\forall\, k\in \N:\ A^{n_k}x=y^{(k)}+\sum_{m=k+1}^\infty B^{n_m-n_k}y^{(m)}.
\]

Furthermore, since, by \eqref{est2}, for each $k\in \N$,
\[
\left\|y^{(k)}-A^{n_k}x\right\|
\le \sum_{m=k+1}^\infty \left\|B^{n_m-n_k}y^{(m)}\right\|
\le \sum_{m=k+1}^\infty
|w|^{-m}
=\dfrac{|w|^{-(k+1)}}{1-|w|^{-1}}\to 0,\ k\to\infty,
\]
which, in view of the denseness of $Y$ in $(X,\|\cdot\|)$, shows that the orbit
$\orb(x,A)$ of $x$ under $A$
is dense in $(X,\|\cdot\|)$ as well. This implies that the operator $A$ is \textit{hypercyclic}.

Now, let us show that the operator $A$ is \textit{chaotic}. Indeed, for any $N\in \N$, let $x_1,\dots,x_N\in \F$ be arbitrary, then

\begin{equation}\label{pp}
\begin{split}
x:=(&x_1,\dots,x_N,w^{-(N+\dots+1)}x_1,\dots,w^{-(2N-1+\dots+N)}x_N,\\
&w^{-(2N+\dots+N+1)}x_1,\dots,w^{-(3N-1+\dots+2N)}x_N,\dots)\in D(A^N)
\end{split}
\end{equation}
and
\[
A^N x=x.
\]

For any 
\[
y:=(y_1,\dots,y_n,0,\dots)\in Y
\]
with some $n\in \N$ and arbitrary
$N\ge n$, consider the periodic point defined by \eqref{pp} with
\[
x_k:=\begin{cases}
y_k,& k=1,\dots,n,\\
0,& k=n+1,\dots,N.
\end{cases}
\]

Then, for $x\in C^\infty(A)$ defined by \eqref{pp},
\[
\|y-x\|\le \sum_{k=1}^{\infty}|w|^{-kN}\|y\|
=\|y\|\dfrac{|w|^{-N}}{1-|w|^{-N}}\to 0,\ N\to \infty.
\]

Whence, in view of the denseness of $Y$ in $(X,\|\cdot\|)$, we infer that the set of periodic points of $A$ is dense in $(X,\|\cdot\|)$ as well,
i.e., the operator $A$ is \textit{chaotic}.

It remains to prove that, if the space $X$ \textit{complex}, the spectrum $\sigma(A)$ of $A$ is the whole complex plane $\C$, with every $\lambda \in \C$ being an \textit{eigenvalue} for $A$ with associated one-dimensional eigenspace. Indeed, let $\lambda\in \C$
be arbitrary. Then for an $x:=(x_k)_{k\in \N}\in X\setminus \{0\}$, the equation
\begin{equation}\label{ev}
Ax=\lambda x
\end{equation}
is equivalent to
\[
(w^k x_{k+1})_{k\in \N}=\lambda(x_k)_{k\in \N},
\]
i.e.,
\[
w^k x_{k+1}=\lambda x_k,\ k\in \N
\]

Whence, we recursively infer that
\[
x_k=\dfrac{\lambda^{k-1}}{w^{1+2+\cdots+k-1}}x_1
=\dfrac{\lambda^{k-1}}{w^{\frac{k(k-1)}{2}}}x_1
=\left(\dfrac{\lambda}{w^{\frac{k}{2}}}\right)^{k-1}x_1,\
k\in \N,
\]
where for $\lambda=0$, $0^0:=1$.

Considering that $|w|>1$, for all sufficiently large $k\in \N$, we have:
\[
\left|\dfrac{\lambda}{w^{\frac{k}{2}}}\right|^{k-1}
=\left(\dfrac{|\lambda|}{|w|^{\frac{k}{2}}}\right)^{k-1}
\le \left(\dfrac{1}{2}\right)^{k-1},
\]
which implies that
\[
(y_k)_{k\in \N}:=\left(\left(\dfrac{\lambda}{w^{\frac{k}{2}}}\right)^{k-1}\right)_{k\in \N}\in X.
\]

Further, since
\[
w^ky_{k+1}=w^k\dfrac{\lambda^{k}}{w^{1+2+\cdots+k}}
=\dfrac{\lambda^{k}}{w^{1+2+\cdots+k-1}}
=\dfrac{\lambda^{k}}{w^{\frac{k(k-1)}{2}}}x_1
=\left(\dfrac{\lambda}{w^{\frac{k-1}{2}}}\right)^{k},\ k\in \N,
\]
we similarly conclude that
\[
(w^ky_{k+1})_{k\in \N}\in X,
\]
and hence,
\[
(y_k)_{k\in \N}\in D(A)\setminus \{0\}.
\]

Thus, we have shown that, for an arbitrary $\lambda \in \C$, all solutions of equation \ref{ev} are the sequences of the form
\[
x_1\left(\left(\dfrac{\lambda}{w^{\frac{k}{2}}}\right)^{k-1}\right)_{k\in \N}\in D(A),
\]
where $x_1\in \C$ is arbitrary. They form the \textit{one-dimensional} subspace of $X$ spanned by the sequence
$\left(\left(\dfrac{\lambda}{w^{\frac{k}{2}}}\right)^{k-1}\right)_{k\in \N}\in D(A)\setminus \{0\}$, which completes the proof.
\end{proof}

\begin{rem}
Considering that, in the (real or complex) separable Banach space
$X:=L_p(0,\infty)$ ($1\le p<\infty$) or $X:=C_0[0,\infty)$ (of continuous on $[0,\infty)$
and vanishing at infinity functions), the latter equipped with the maximum norm
\[
C_0[0,\infty)\ni x\mapsto\|x\|_\infty:=\max_{t\ge 0}|x(t)|,
\]
there similarly exist countable dense subsets of
\textit{eventually zero} elements (see, e.g., \cite{Dun-SchI}), the proof of the prior theorem 
can be modified for the weighted backward shift
\[
[Ax](t):=w^tx(t+a),\ t\ge 0,
\]
where $w\in \F$ ($\F:=\R$ or $\F:=\C$) such that $|w|>1$ is arbitrary, with the maximal domain
\[
D(A):=\left\{x(\cdot)\in X\,\middle|\,w^\cdot x(\cdot+a)\in X \right\}.
\]

In the case of $L_p(0,\infty)$ ($1\le p<\infty$), the notations $x(\cdot)$ and $w^\cdot x(\cdot+a)$ are used to designate both the equivalence classes of functions and their corresponding representatives. 
\end{rem}

\section{Acknowledgments}

My utmost appreciation goes to Professor Karl Grosse-Erdmann of the Universit\'e de Mons Institut de Math\'ematique for his kind interest in my recent endeavors in the realm of unbounded linear hypercyclicity and chaos (see also \cite{Markin2018(8)}), turning my attention to the existing research on the subject, and communicating relevant references. I am also very grateful to my pupil, Mr. Edward Sichel of California State University, Fresno, for his kind and thoughtful assistance.

 

\begin{thebibliography}{99}
\bibitem{B-Ch-S2001}
{J. B\`es, K.C. Chan, and S.M. Seubert},	
\textit{Chaotic unbounded differentiation operators},	
{Integral Equations and Operator Theory}	
\textbf{40}	
{(2001)},	
{no.~3},	
{257--267}.	
\bibitem{deL-E-G-E2003}
{R. deLaubenfels, H. Emamirad, and K.-G. Grosse-Erdmann},	
\textit{Chaos for 
semigroups of unbounded operators},	
{Math. Nachr.}	
\textbf{261/262}	
{(2003)},	
{47--59}.	
\bibitem{Dunford1954}
{N. Dunford},	
\textit{Spectral operators},	
{Pacific J. Math.}	
\textbf{4}	
{(1954)},	
{321--354}.	
\bibitem{Survey58}
{\bysame},	
\textit{A survey of the theory of spectral operators}, 
{Bull. Amer. Math. Soc.}	
\textbf{64}	
{(1958)},	
{217--274}.	
\bibitem{Dun-SchI}
{N. Dunford and J.T. Schwartz with the assistance of W.G. Bade and R.G. Bartle},	
\textit{Linear Operators. Part I: General Theory},	
{Interscience Publishers},	
{New York},		
{1958}.		
\bibitem{Dun-SchII}
{\bysame},	
\textit{Linear Operators. Part II: Spectral Theory. Self Adjoint Operators in Hilbert Space}, 
{Interscience Publishers},	
{New York},		
{1963}.		
\bibitem{Dun-SchIII}
{\bysame},	
\textit{Linear Operators. Part III: Spectral Operators}, 
{Interscience Publishers},	
{New York},		
{1971}.		
\bibitem{Emam-Hesh2005}
{H. Emamirad and G.S. Heshmati},	
\textit{Chaotic weighted shifts in Bargmann space},	
{J. Math. Anal. Appl.}	
\textbf{308}	
{(2005)},	
{36--46}.	
\bibitem{Godefroy-Shapiro1991}
{G. Godefroy and J.H. Shapiro},	
\textit{Operators with dense, invariant, cyclic vector manifolds},	
{J. Funct. Anal.}	
\textbf{98}	
{(1991)},	
{229--269}.	
\bibitem{Grosse-Erdmann-Manguillot}
{K.-G. Grosse-Erdmann and A.P. Manguillot},	
\textit{Linear Chaos}, 
{Universitext},	
{Springer-Verlag},	
{London},		
{2011}.		
\bibitem{Guirao-Montesinos-Zizler}
{A.J. Guirao, V. Montesinos, and V. Zizler},	
\textit{Open Problems in the Geometry and Analysis of Banach Spaces}, 
{Springer International Publishing},	
{Switzerland},		
{2016}.		
\bibitem{Markin2018(7)}
{M.V. Markin},	
\textit{On the non-hypercyclicity of scalar type spectral operators and collections of their exponentials},	
{Proceedings of the 9th International Conference on Topological Algebras and Their Applications},	
{Nova Science Publishers, Inc.},	
{New York}		
{(to appear)}, \href{https://arxiv.org/abs/1807.07423}{\textcolor{blue}{arXiv:1807.07423}}.
\bibitem{Markin2018(8)}
{\bysame},	
\textit{On general construct of chaotic unbounded linear operators in Banach spaces with Schauder bases},	
\href{https://arxiv.org/abs/1812.02294}{\textcolor{blue}{arXiv:1812.02294}}.
\bibitem{Plesner}
{A.I. Plesner},	
\textit{Spectral Theory of Linear Operators},	
{Nauka},	
{Moscow},		
{1965}		
{(Russian)}.
\bibitem{Rolewicz1969}
{S. Rolewicz},	
\textit{On orbits of elements},	
{Studia Math.}	
\textbf{32}	
{(1969)},	
{17--22}.	
\end{thebibliography}
\end{document}